\documentclass{amsart}
\usepackage{graphics}

\usepackage{amsmath, amsfonts,amssymb,amsopn,amscd,amsthm}
\usepackage{mathrsfs}
 \newtheorem{theorem}{Theorem}[section]
 \newtheorem{cor}[theorem]{Corollary}
 \newtheorem{lem}[theorem]{Lemma}{\rm}
 \newtheorem{prop}[theorem]{Proposition}

 \newtheorem{defn}[theorem]{Definition}{\rm}

 \newtheorem{ex}{Example}
\numberwithin{equation}{section}

\usepackage{amsfonts,amsmath}

\def\x{\mathbf{x}}
\def\B{\mathbf{B}}

\def\R{\mathbb{R}}

\def\N{\mathbb{N}}

\def\Q{\mathbf{Q}}
\def\M{\mathbf{M}}

\def\C{\mathbb{C}}

\def\bI{\mathbf{I}}
\def\P{\mathbf{P}}
\def\Z{\mathbb{Z}}

\def\A{\mathbf{A}}

\def\B{\mathbf{B}}

\def\I{\mathbf{I}}

\def\G{\mathbf{G}}

\def\z{\mathbf{z}}
\def\v{\mathbf{v}}
\def\y{\mathbf{y}}

\def\bI{\mathbf{I}}
\def\y{\mathbf{y}}

\def\g{\mathbf{g}}

\def\u{\mathbf{u}}
\def\p{\mathbf{p}}

\def\blambda{\mathbf{\lambda}}
\def\s{\mathcal{S}}

\def\n{\mathcal{N}}

\def\vol{{\rm vol}\,}

\def\hom{\mathbf{Hom}_{d}}
\def\homq{\mathbf{Hom}^q_d}

\title{unit  balls of constant volume: which one has optimal representation?}
\begin{document}
\author{Jean B. Lasserre}
\address{LAAS-CNRS and Institute of Mathematics\\
University of Toulouse\\
LAAS, 7 avenue du Colonel Roche\\
31077 Toulouse C\'edex 4, France\\
Tel: +33561336415}
\email{lasserre@laas.fr}

\date{}
\begin{abstract}
In the family of unit balls with constant volume we look at
the ones  whose algebraic representation
has some extremal property.
We consider the family of nonnegative homogeneous polynomials of even degree $d$ whose sublevel set
$\G=\{\x: g(\x)\leq 1\}$ (a unit ball) has same fixed volume and want to find in this family the one 
that minimizes either the $\ell_1$-norm or the $\ell_2$-norm
of its vector of coefficients. Equivalently, among all degree-$d$ polynomials 
of constant $\ell_1-$ or $\ell_2$-norm, which one minimizes the volume of its level set $\G$.
We first show that in both cases this is a convex optimization problem
with a unique optimal solution $g^*_1$ and $g^*_2$ respectively. We also show that 
$g^*_1$ is the $L_p$-norm polynomial $\x\mapsto\sum_{i=1}^n x_i^{p}$,
thus recovering a parsimony property of the $L_p$-norm via
$\ell_1$-norm minimization. (Indeed $n=\Vert g^*_1\Vert_0$ is the minimum number 
of non-zero coefficient for $\G$ to have finite volume.)
This once again illustrates the power and versatility of the $\ell_1$-norm relaxation
strategy in optimization when one searches for an optimal solution with parsimony properties.
Next we show that $g^*_2$ is not sparse at all (and so differs from $g^*_1$) but is still a sum of $p$-powers of linear forms. We also characterize the 
unique optimal solution of the same problem where one searches for an SOS homogeneous polynomial
that minimizes the trace of its associated (psd) Gram matrix, hence 
aiming at finding a solution which is a sum of a few squares only.
Finally, we also extend these  results to generalized homogeneous polynomials,
which includes $L_p$-norms when $0<p$ is rational.
\end{abstract}

\maketitle

\section{Introduction}

It is well-known that the shape of the Euclidean unit ball
$\B_2=\{\,\x:\sum_{i=1}^nx_i^2\leq1\,\}$
has spectacular geometric properties with respect to other shapes. For
instance, the sphere has the smallest surface area among all surfaces enclosing a given volume and it encloses the largest volume among all closed surfaces with a given surface area; Hilbert and Cohn-Vossen \cite{hilbert} even describe eleven geometric properties of the sphere!

But $\B_2$ has also another spectacular (non-geometric) property related to its algebraic representation which is 
obvious even to people with a little background in Mathematics: 
Namely, its defining polynomial 
$\x\mapsto g_2(\x):=\sum_{i=1}^nx_i^2$ cannot be simpler!! Indeed, among all nonnegative quadratic  homogeneous polynomials
$\x\mapsto g(\x)=\sum_{i\leq j}g_{ij}x_ix_j$ that define a bounded ball $\{\,\x:g(\x)\leq 1\,\}$,
$g_2$ is the one that minimizes the ``cardinality norm" $\Vert g\Vert_0:=\#\{\,(i,j):g_{ij}\neq0\,\}$ (which actually is not a norm).
Only $n$ coefficients of $g_2$ do not vanish and there cannot be less than $n$ non zero coefficients to define a bounded ball $\{\,\x: g(\x)\leq1\,\}$. The same is true for the $d$-unit ball $\B_d=\{\,\x:\sum_{i=1}^nx_i^d\leq1\,\}$ and its defining polynomial
$\x\mapsto g_d(\x)=\sum_ix_i^d$ for any even integer $d$,
when compared to any other nonnegative homogeneous polynomial $g$ of degree $d$ whose 
sublevel set $\{\,\x:g(\x)\leq1\,\}$ has finite Lebesgue volume. Indeed, again $\Vert g_d\Vert_0=n$, i.e.,
out of potentially ${n+d-1\choose d}$ coefficients only $n$ do not vanish! In other words,
\begin{equation}
\label{norm-ell0}
g_d\,=\,\displaystyle\arg\min_g\,\{\,\Vert g\Vert_0:\: {\rm vol}\,(\{\x:g(\x)\leq1\,\})\,\leq\,1\,\}\end{equation}
where the minimum is taken over all homogeneous polynomials of degree $d$. 

So an natural question which arises is as follows:
{\it In view of the many ``geometric properties" of the unit ball $\B_d$, is 
the  ``algebraic sparsity" of its representation $\{\x:\sum_ix_i^d\leq1\,\}$ a coincidence or does it also corresponds 
to a certain extremal property on all possible representations?}

So we are interested in the following optimization problem in computational geometry
and with an algebraic flavor. 

{\it Given an even integer $d$, determine the nonnegative homogeneous polynomial
$g^*$ of degree $d$ whose $\ell_1$-norm $\Vert g^*\Vert_1$ (or $\ell_2$-norm $\Vert g^*\Vert_2$)
of its vector of coefficients is minimum among all degree-$d$ nonnegative homogeneous polynomials
with same (fixed) volume of their sublevel set $\G=\{\,\x:g(\x)\leq 1\,\}$. That is, solve:}
\begin{equation}
\label{relax-0-1}
\inf_{g}\,\{\,\Vert g\Vert_{p=1,2}:\: {\rm vol}\,(\G)\,=\,1\,;\:\mbox{$g$ homogeneous of degree $d$}\,\}.\end{equation}
In particular, {\it Can the parsimony property
of the $L_d$-unit balls be recovered from (\ref{relax-0-1}) with the $\ell_1$-norm $\Vert g\Vert_1$
(instead of minimizing the nasty function $\Vert \cdot\Vert_0$ in (\ref{norm-ell0}))?}\\

By homogeneity, this problem also has the equivalent formulation:  {\it Among all homogeneous polynomials
$g$ of degree $d$ and with constant norm $\Vert g\Vert_1=1$ (or $\Vert g\Vert_2=1$) find the one 
with level set $\G$ of minimum volume.}\\

One goal of this paper is to prove that (\ref{relax-0-1}) is a convex optimization problem with a {\it unique optimal solution},
which is the same as $g_d$ in (\ref{norm-ell0}) when one minimizes the $\ell_1$-norm $\Vert g\Vert_1$.
In addition $g_d$ cannot be an optimal solution of (\ref{relax-0-1}) when one minimizes 
the $\ell_2$-norm $\Vert g\Vert_2$ (except when $d=2$).
This illustrates in this context of  computational geometry that again,
the sparsity-induced $\ell_1$-norm does a perfect job
in the relaxation (\ref{relax-0-1}) (with $\Vert\cdot\Vert_1$) of problem (\ref{norm-ell0}) with $\Vert \cdot\Vert_0$.
This convex ``relaxation trick" 
in (non convex) $\ell_0$-optimization has been used successfully in several important applications; see e.g. 
Cand\`es et al. \cite{candes}, Donoho \cite{donoho1}, Donoho and Elad \cite{donoho2}
in compressed sensing applications and Recht et al. \cite{recht} for matrix applications (where the 
small-rank induced {\it nuclear norm} is the matrix analogue of the $\ell_1$-norm). For more details on optimization with sparsity constraints and/or sparsity-induced penalties,
the interested reader is referred to Beck and Eldar \cite{beck} and Bach et al. \cite{bach}.\\

To address our problem we consider the following framework: 
Let $\hom\subset\R[\x]_{d}$ be the vector space of homogeneous polynomials of even degree $d$,
and given $g\in\hom$, let $\g=(g_\alpha)$ be its vector of coefficients, i.e.,
\[\x\mapsto g(\x)\,=\,\sum_\alpha g_\alpha\,\x^\alpha \:\left(=\sum_{\alpha}g_\alpha\,\x_1^{\alpha_1}\cdots \x_n^{\alpha_n}\right),\qquad \sum_i\alpha_i=d,\]
with standard $\ell_1$-norm $\Vert g\Vert_1=\vert \g\vert=\sum_{\alpha}\vert g_\alpha\vert$. With any 
$g\in\hom$ is associated its sublevel set $\G\subset\R^n$ defined by:
\begin{equation}
\label{set-G}
\G\,:=\,\{\,\x\in\R^n\::\: g(\x)\,\leq\,1\,\},\qquad g\in\hom.
\end{equation}
In particular, with $\x\mapsto g^*(\x):=\sum_{i=1}^nx_i^{d}$, the sublevel set $\G^*$ is nothing less than
the standard $d$-unit ball
\[\B_{d}\,=\,\{\,\x:\sum_{i=1}^nx_i^{d}\leq1\,\}\,=:\,\{\,\x:\Vert \x\Vert^{d}_{d}\,\leq1\,\},\]
whose Lebesgue volume ${\rm vol}\,(\B_d)$ is denoted $\rho_d$.
(When $g\in\hom$ is convex then $\x\mapsto g(\x)$ defines a norm $\Vert\x\Vert_g:=g(\x)^{1/d}$ with $\G$ as associated unit ball.)

\subsection*{Contribution} 

(a) In a first contribution we prove that 
the optimization problem:
\begin{equation}
\label{pb1-intro}
\P_1:\quad \inf_g\,\{\,\Vert g\Vert_1\,:\: {\rm vol}\,(\G)\,\leq\,\rho_d\,;\quad g\in\hom\,\},\end{equation}
has a unique optimal solution $g^*_1$ which is the $L_d$-norm polynomial $\x\mapsto \Vert\x\Vert_{d}^{d}$.
Observe that $g^*_1$ has the minimal number $n$ of coefficients 
over potentially $s(d):={n-1+d\choose d}$  coefficients. (Indeed for a polynomial $g\in\hom$
with $m<n$ non zero coefficients, its sublevel set $\G$ cannot have finite Lebesgue volume.)
Therefore the $L_d$-norm polynomial $g^*_1$ associated with the unit ball $\B_d$ is the ``sparsest" solution among all $g\in\hom$ such that $\vol(\G)\leq\vol(\B_d)$.
In particular, $g^*_1$ not only solves problem $\P_1$ but also solves the non convex optimization problem
\[\P_0:\quad \inf_g\,\{\,\Vert g\Vert_0\,:\: {\rm vol}\,(\G)\,\leq\,\rho_d\,;\quad g\in\hom\,\},\]
of which $\P_1$ is a ``convex relaxation". But this is also equivalent to state that among all
homogeneous polynomials of degree $d$ with constant $\ell_1$-norm,
the $L_d$-unit ball is the one with minimum volume ${\rm vol}\,(\B_d)\leq{\rm vol}(\G)$.

(b) In a second contribution we consider the $\ell_2$-norm version of (\ref{pb1-intro}):
\begin{equation}
\label{pb2-intro}
\P_2:\quad \inf_g\,\{\,\Vert g\Vert_2\,:\: {\rm vol}\,(\G)\,\leq\,\rho_d\,;\quad g\in\hom\,\},\end{equation}
with weighted Euclidean norm $g\mapsto \Vert g\Vert_2$ defined by:
\[\Vert g\Vert_2^{2}\,:=\,\sum_{\vert\alpha\vert=d} c_\alpha\,g_\alpha^2,\qquad g\in\hom,\quad
\mbox{where }c_\alpha:=\frac{(d){\rm !}}{\alpha_1{\rm !}\cdots\alpha_n{\rm !}}.\]
We then show that $\P_2$ also has a unique optimal solution $g^*_2$,
but in contrast to the optimal solution $g^*_1$ of problem $\P_1$, $g^*_2$ is not sparse at all! This is because
one can show that all ${n-1+d\choose d}$ coefficients of the form $g^*_{2\beta}$ with $\vert\beta\vert=d/2$ are non-zero.
In addition,
$g^*_2$ is a particular sum of squares (SOS) polynomial as it is a sum of $d$-powers of linear forms.
(Notice that $g^*_1$ is also a (very particular and simple) sum of $d$-powers of linear forms.) In particular, when $d=4$
the optimal solution of $\P_2$ is the Euclidean ball $\{\x:\sum_ix_i^2\leq1\}$ which has the equivalent quartic representation
$\{\x:(\sum_{i=1}^nx_i^2)^2\leq1\}$ and the SOS quartic polynomial $\x\mapsto (\sum_{i=1}^nx_i^2)^2$ solves $\P_2$.\\

(c) We also consider the SOS (sum of squares) version of $\P_1$, that is one now searches for
a degree-$d$ SOS homogeneous polynomial $g_\Q(\x)=\v_{d/2}(\x)\Q\v_{d/2}(\x)$, $\Q\succeq0$,
(where $\v_{d/2}(\x)=(\x^\alpha)$, $\vert\alpha\vert=d/2$). That is, one characterizes the unique optimal 
solution of the optimization problem:
\begin{equation}
\label{pb3-intro}
\P_3:\quad \inf_{\Q\succeq0}\,\{\,{\rm trace}\,(\Q)\,:\: {\rm vol}\,(\G_{\Q})\,\leq\,\rho_d\,;\: \Q\succeq0\,\}.\end{equation}
In this matrix context, ${\rm trace}\,(\Q)$ is the {\it nuclear norm} of $\Q$ and so
solving $\P_3$ aims at finding an optimal solution $\Q^*$ with small rank, which translates into an homogeneous
polynomial  $g_{\Q^*}$  which is a sum of a few squares. We also proves that
$g^*_1$ associated with the $L_d$-unit ball cannot be an optimal solution of $\P_3$ (and indeed
$g_{\Q^*}$ being a sum of a few squares does not necessarily implies
that it has a small number of coefficients).

(d) Finally we also show that results in (a) and (b) extend to the case of 
other values of $d$ (including $p=1$ and rationals) in which case one now deals with positively homogeneous 
``generalized polynomials" (instead of homogeneous polynomials) and one has to define an appropriate finite-dimensional
analogue analogue of $\hom$. This includes the interesting case of 
the $L_1$-unit ball $\{\,\x:\sum_i\vert x_i\vert\leq1\,\}$ and when $p<1$, balls which are not associated with norms.

\section{Notation, definitions and preliminary results}
\label{notation}
Let $\R[\x]$ denote the ring or real polynomials in the variables $\x=(x_1,\ldots,x_n)$, and 
let $\R[\x]_d$ be the vector space of real polynomials of degree at most $d$.
Similarly, let $\Sigma[\x]\subset\R[\x]$ denote the convex cone of real polynomials that are sums of squares (SOS) of polynomials,
and $\Sigma[\x]_d\subset\Sigma[\x]$ its subcone of SOS polynomials of degree at most $d$.
Denote by $\s^m$ the space of $m\times m$ real symmetric matrices. For a given matrix $\A\in\s^m$, the notation
$\A \succeq 0$ (resp. $\A\succ0$) means that $\A$ is positive semidefinite (psd) (resp. positive definite (pd)), i.e., all its eigenvalues 
are real and nonnegative (resp. positive). 

A polynomial $p\in\R[\x]_d$ is homogenous if $p(\lambda \x)=\lambda^d p(\x)$ for all $\x\in\R^n$, $\lambda\in\R$.
A function $f:\R^n\to\R$ is positively homogeneous of degree $d\in\R$ if $f(\lambda\x)=\lambda^df(\x)$
for all $0\neq\x\in\R^n$, $\lambda>0$. For instance $x\mapsto \vert x\vert$ is not homogeneous but is positively homogeneous of degree $1$.

Let $\hom\subset\R[\x]_{d}$ be the vector space of homogeneous polynomials of even degree $d$, and
let $\N^n_{d}:=\{(\alpha_1,\ldots,\alpha_n)\::\:\sum_i\alpha_i=d\}$.
For an homogeneous polynomial $g\in\R[\x]_{d}$, and
with $s(d):={n-1+d\choose d}$), let  $\g=(g_\alpha)\in\R^{s(d)}$ be its vector of coefficients, i.e.,
\[\x\mapsto g(\x)\,:=\,\sum_{\alpha\in\N^n_{d}}g_\alpha\,\x^\alpha
\:\left(=\sum_{\alpha\in\N^n_{d}}g_\alpha\,x_1^{\alpha_1}\cdots x_n^{\alpha_n}\right).\]
Denote by $\G\subset\R^n$
its associated sublevel set $\G:=\{\x:g(\x)\leq1\}$. 

Let $\P[\x]_d\subset\hom$ be the set of homogeneous polynomials of degree $d$
whose associated level set $\G$ has finite Lebesgue volume. It is a convex cone; see \cite[Proposition 2.1]{lowner}.
Let $f:\hom\to \R_+$ be the function defined by:
\[g\,\mapsto\,f(g)\,:=\,\left\{\begin{array}{ll}
{\rm vol}\,(\G)&\mbox{if $g\in\P[\x]_d$}\\
+\infty&\mbox{otherwise.}\end{array}\right.\]
It is important to realize that the sublevel set $\G$ need not be convex! For instance
Figure \ref{figure1} displays two examples of non convex sets $\G$.
\begin{center} 
\begin{figure}
\label{figure1}
\resizebox{0.9\textwidth}{!}
%\resizebox{\textwidth}{!}
{\includegraphics{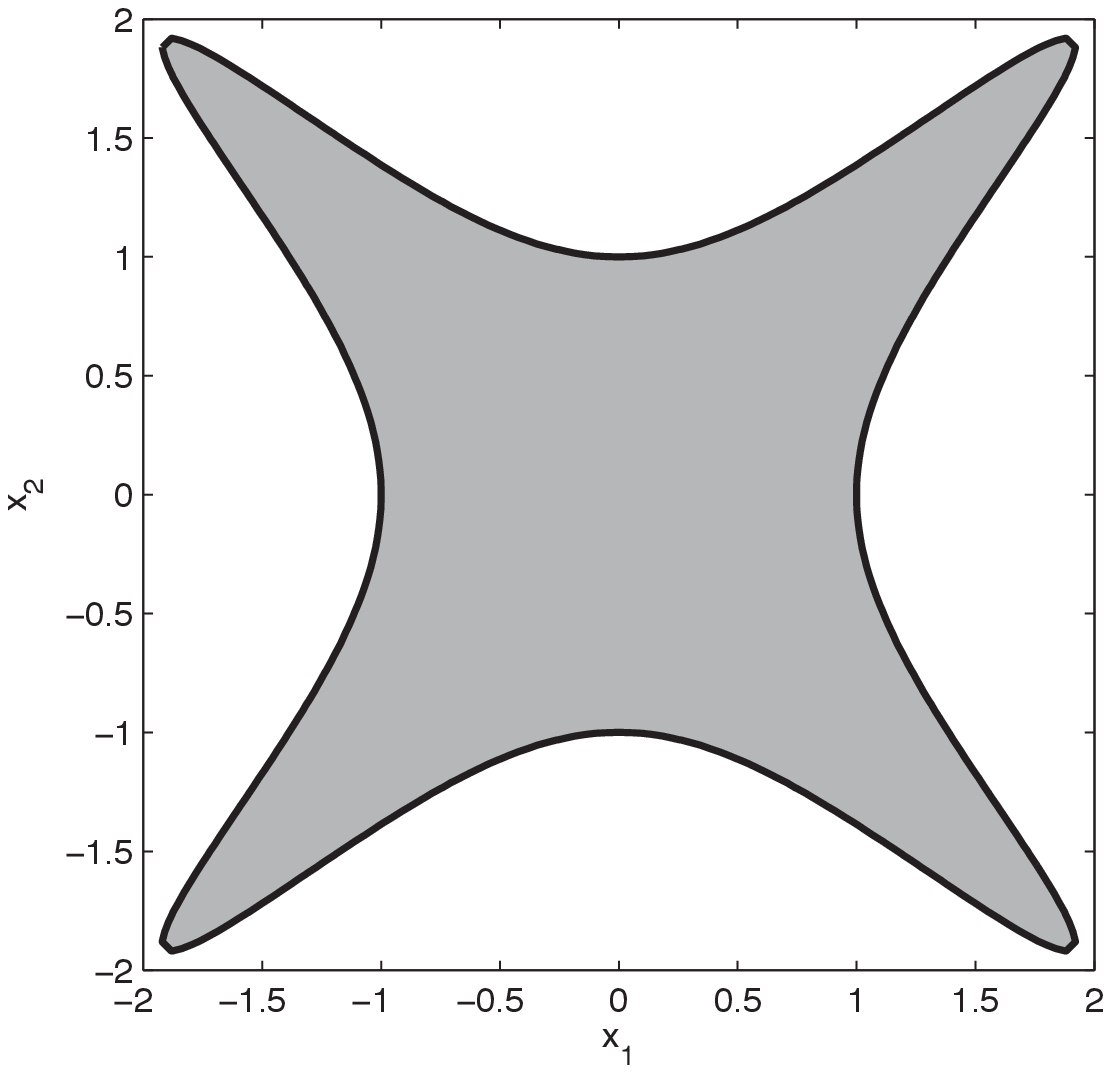}\includegraphics{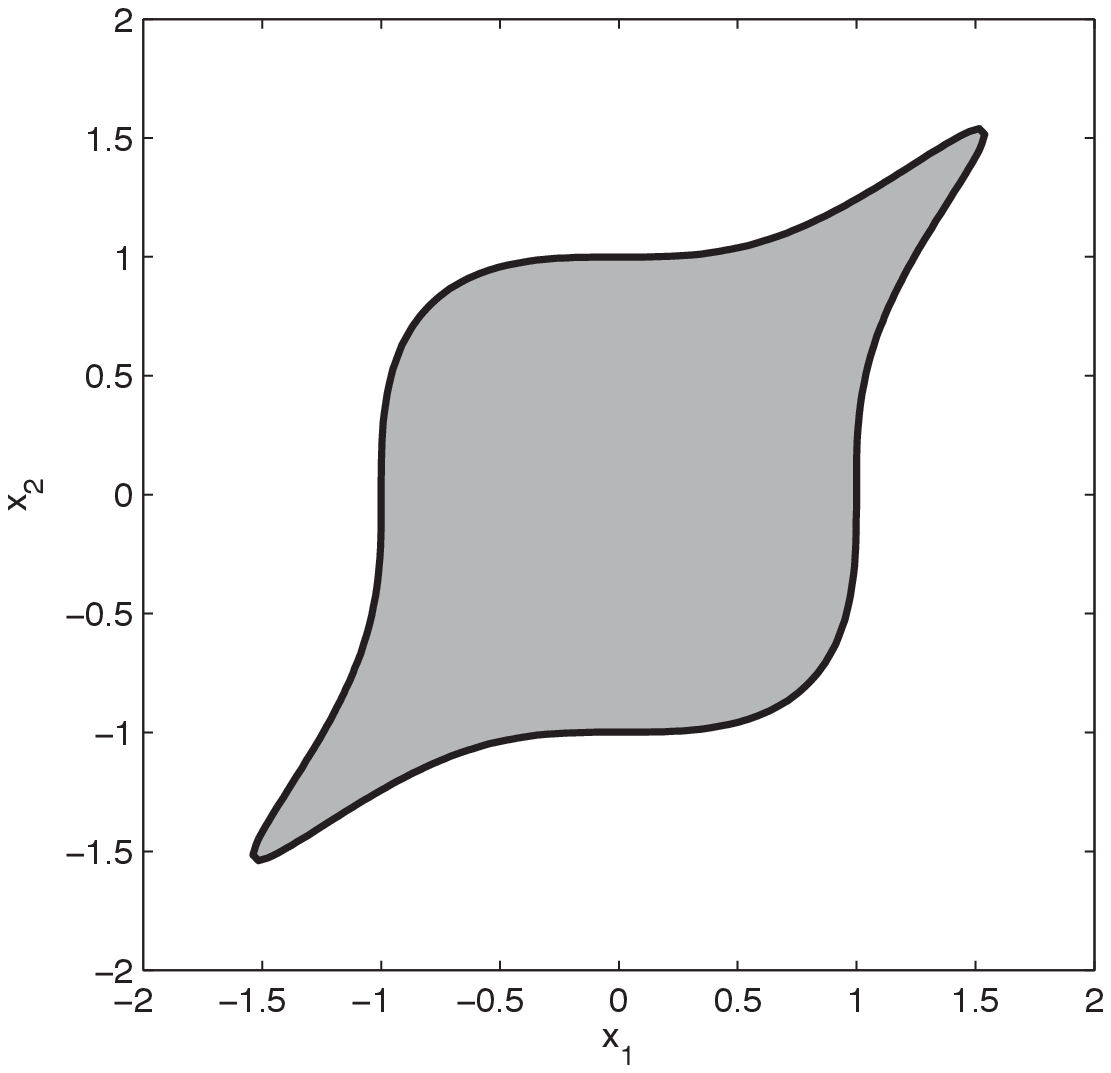}}
%\caption{$n=2$: $\G$ with $g(\x)=x^2+y^2-1.925(x^4+y^4)$}
\caption{$\G_1$ with $x^4+y^4-1.925\,x^2y^2$ and $x^6+y^6-1.925\,x^3y^3$}
\end{figure}
\end{center}
More generally, for every $\alpha\in\N^n$ define $f_\alpha:\P[\x]_d\to\R$ by
\[g\,\mapsto\,f_\alpha(g)\,:=\,\left\{\begin{array}{ll}
\displaystyle\int_{\G}\x^\alpha\,d\x&\mbox{if $g\in\P[\x]_d$}\\
+\infty&\mbox{otherwise.}\end{array}\right.\]
In particular, $f(g)=f_0(g)$.

\subsection*{A preliminary result}

We will need the following result of independent interest already proved in \cite{lowner} but 
for which we provide a brief sketch. We use the same technique based on Laplace transform
as in Lasserre \cite{monthly} and Lasserre and Zeron \cite{applicationes}  for providing closed form expressions for certain class of integrals.

\begin{theorem}
\label{th1}
Let $h:\R^n\to\R$ be a nonnegative positively homogeneous function of degree $0\neq d\in\R$ such that ${\rm vol}\,(\{\,\x:h(\x)\leq 1\,\})<\infty$.
Then for every $\alpha\in\N^n$:
\begin{equation}
 \label{vol-alpha}
\int_{\{\,\x:h(\x)\leq1\,\}}\x^\alpha\,d\x\,=\,\frac{1}{\Gamma(1+(n+\vert\alpha\vert)/d)}\int_{\R^n}\x^\alpha\,\exp(-h(\x))\,d\x.
\end{equation}
In particular when $d$ is an even integer: For every $g\in\P[\x]_d$,% and every $\alpha\in\N^n$,
 \begin{equation}
 \label{vol-exp}
f(g)\,=\,{\rm vol}\,(\G)\,=\,\frac{1}{\Gamma(1+n/d)}\int_{\R^n}\exp(-g(\x))\,d\x,
 \end{equation}
and the function $f$ is nonnegative, strictly convex and homogeneous of degree $-n/d$.
Moreover, if $g\in\,{\rm int}\,(\P[\x]_d)$:

\begin{eqnarray}
\label{relation-1}
\frac{\partial f(g)}{\partial g_\alpha}&=&-\frac{n+d}{d}\,\int_{\G}\x^\alpha\,d\x,\qquad\alpha\,\in\,\N^n_d\\
\label{relation-2}
\int_\G g(\x)\,d\x&=&\frac{n}{n+d}\,\int_{\G}\,d\x.
\end{eqnarray}
\end{theorem}
\begin{proof}
For $\alpha\in\N^n$, let $v_\alpha:\R_+\to\R$ be the function $y\mapsto v_\alpha(y):=\int_{\{\x:h(\x)\leq\,y\,\}}\x^\alpha\,d\x$.
Observe that $v_\alpha(y)=0$ whenever $y<0$. So let $\mathcal{L}[v_\alpha]:\C\to\C$ be the Laplace transform $\mathcal{L}[v_\alpha]$ of the function $v_\alpha$, i.e.,
\[\lambda\,\mapsto\,\mathcal{L}[v_\alpha](\lambda)\,:=\,\int_0^\infty \exp(-\lambda y)\,v_\alpha(y)\,dy,\qquad \lambda\in\C\,;\:\Re(\lambda)>0.\]
Let $H_\alpha:\R_+\to\R$ be the function $\lambda\mapsto H_\alpha(\lambda):=\mathcal{L}[v_\alpha](\lambda)$, $\lambda\in\R_+$. Then:
\begin{eqnarray*}
H_\alpha(\lambda)&=&\int_0^\infty \exp(-\lambda y)\left(\int_{\{\,\x:h(\x)\leq y\,\}}\x^\alpha\,d\x\right)\,dy\\
&=&\int_{\R^n}\x^\alpha\left(\int_{h(\x)}^\infty \exp(-\lambda y)\,dy\right)\,d\x\\
&=&\frac{1}{\lambda}\int_{\R^n}\x^\alpha\exp(-\lambda h(\x))\,d\x\\
&=&\frac{1}{\lambda^{1+(n+\vert\alpha\vert)/d}}\int_{\R^n}\x^\alpha\exp(-h(\x))\,d\x\quad\mbox{[by homogeneity]}\\
&=&\frac{\Gamma(1+(n+\vert\alpha\vert)/d)}{\lambda^{1+(n+\vert\alpha\vert)/d}}\left(\underbrace{\frac{1}{\Gamma(1+(n+\vert\alpha\vert)/d)}\int_{\R^n}\x^\alpha\exp(-h(\x))\,d\x}_{c}\,\right)\\
&=&c\,\frac{\Gamma(1+(n+\vert\alpha\vert)/d)}{\lambda^{1+(n+\vert\alpha\vert)/d}}.
\end{eqnarray*}
%&=&\mathcal{L}[y^{n/d}]\,\left(\underbrace{\frac{1}{\Gamma(1+n/d)}\int_{\R^n}\exp(-g(\x))\,d\x}_{c}\,\right).
%\end{eqnarray*}
The function $H_\alpha$ is analytic on $\Theta=\{\lambda\in\C:\Re(\lambda) >0\}$ and
coincides with $\mathcal{L}[v_\alpha]$ on $[0,\infty)$. By the identity Theorem on analytic functions
(see e.g. \cite[Theorem III.3.2, p. 125]{freitag}),
$H_\alpha=\mathcal{L}[v_\alpha]$ on $\Theta$. But $H_\alpha=\mathcal{L}[c\,y^{(n+\vert\alpha\vert)/d}]$ which concludes the proof and yields (\ref{vol-alpha}) when $y=1$. 

Next when $g\in{\rm int}\,(\P[\x]_d)$, one obtains (\ref{relation-1}) by differentiating under the integral sign which is 
permitted in this context; see \cite{lowner} for a rigorous proof. This yields
\begin{eqnarray*}
\frac{\partial f(g)}{\partial g_\alpha}&=&\frac{-1}{\Gamma(1+n/d)}\int_{\R^n}\x^\alpha\,\exp(-g(\x))\,d\x,\qquad\alpha\in\N^n_d\\
&=&\frac{-\Gamma(1+(n+d)/d)}{\Gamma(1+n/d)}\int_{\G}\x^\alpha\,d\x,\qquad\alpha\in\N^n_d\\
&=&-\,\frac{n+d}{d}\int_{\G}\x^\alpha\,d\x,\qquad\alpha\in\N^n_d,\end{eqnarray*}
where we have use the identity $\Gamma(z+1)=z\,\Gamma(z)$.
Finally, to get
(\ref{relation-2}) observe that $f$ is a positively homogeneous function of degree $-n/d$ and so Euler's identity 
$\langle \nabla f(g),g\rangle=-n\,f(g)/d$ for homogeneous functions yields:
\[-\frac{n}{d}\int_{\G}d\x\,=\,\langle g,\nabla f(g)\rangle\,=\,-\frac{n+d}{d}\int_{\G}g(\x)\,d\x.\]
\end{proof}

\section{The $\ell_1$-norm formulation}
\label{section-ell-1}
With $d\in\N$ a fixed even integer and $g\in\hom$ written as 
\[\x\mapsto g(\x)\,=\,\sum_{\alpha\in\N^n_d}\,g_\alpha \,\x^\alpha,\qquad\x\in\R^m,\]
let $\B_d\subset\R^n$ be the $L_d$-unit ball and $\rho_d$ its Lebesgue volume, i.e.,
\begin{equation}
\label{vol-unit}
\B_d\,=\,\{\,\x:\sum_{i=1}^nx_i^{d}\,\leq\,1\,\}\quad \mbox{and}\quad \rho_d\,:=\,{\rm vol}\,(\B_d)\,=\,\int_{\B_d}d\x.
\end{equation}
Let $\Vert g\Vert_1:=\sum_{\alpha\in\N^n_d}\vert g_\alpha\vert$ and consider 
the optimization problem $\P_1$:
\begin{equation}
\label{problem-P-222}
\P_1:\quad\inf_g\:\{\,\Vert g\Vert_1: \:f(g)\,=\,\rho_d\,;\quad g\in\hom\,\}.
\end{equation}
That is, among all degree-$d$ homogeneous polynomials $g$ whose
level set $\G$ has same Lebesgue volume as the $L_d$-unit ball $\B_d$,
one seeks the one which minimizes the $\ell_1$-norm
of its coefficients. 

In fact, since $f(\lambda g)=\lambda^{-n/d}f(g)$ one may replace the constraint
$f(g)=\rho_d$ with the inequality constraint $f(g)\leq\rho_d$ and (\ref{problem-P-222}) reads:
\begin{equation}
\label{problem-P-22}
\P_1:\qquad \ell_1^*\,=\,\inf_g\:\{\,\Vert g\Vert_1: \:f(g)\,\leq\,\rho_d\,;\quad g\in\hom\:\}.
\end{equation}
\begin{theorem}
\label{th2-ell1}
Let $d\geq 2$ be an even integer. The $L_d$-norm polynomial
\[\x\mapsto g^*(\x)\,=\,\Vert\x\Vert_d^{d}\,\left(=\,\sum_{i=1}^n x_i^{d}\,\right),\]
is the unique optimal solution of Problem $\P_1$ in (\ref{problem-P-22}) and
moreover,
\begin{equation}
\label{relationship}
{\rm vol}\,(\B_d)=\displaystyle\int_{\G^*}\,d\x\,=\,\frac{2^n\Gamma(1/d)^n}{n\,d^{n-1}\,\Gamma(n/d)};
\:\displaystyle\int_{\G^*} x_i^{d}\,d\x
\,=\,\frac{{\rm vol}\,(\B_d)}{n+d}\,\quad i=1,\ldots,n.
\end{equation}
\end{theorem}
\begin{proof}
Problem $\P_1$ has an optimal solution $g^*\in\hom$. Indeed let $(g_k)$, $k\in\N$, be a minimizing sequence with
$\Vert g_k\Vert_1\to\ell_1^*\geq0$ as $k\to\infty$. Hence the sequence $(g_n)$ is
$\ell_1$-norm bounded and therefore there is a subsequence $(k_t)$ and a polynomial
$g^*\in\hom$ such that for every $\alpha\in\N^n_d$,
$(g_{k_t})_\alpha\to g^*_\alpha$ as $t\to\infty$. Then of course one also obtains the pointwise
$g_{k_t}(\x)\to g^*(\x)$, as $t\to\infty$.
Next, as $f$ is nonnegative, by Fatou's Lemma
\begin{eqnarray*}
\rho_d\,\geq\,\liminf_{t\to\infty}f(g_{k_t})&=&\liminf_{t\to\infty}\frac{1}{\Gamma(1+n/d)}\int_{\R^n}\exp(-g_{k_t}(\x))\,d\x\\
&\geq&\frac{1}{\Gamma(1+n/d)}\int_{\R^n}\liminf_{t\to\infty}\exp(-g_{k_t}(\x))\,d\x\\
&=&\frac{1}{\Gamma(1+n/d)}\int_{\R^n}\exp(-g^*(\x))\,d\x
\end{eqnarray*}
which proves that $g^*$ is feasible for $\P_1$ and so is an optimal solution of $\P_1$.

Problem $\P_1$ has the equivalent formulation:
\[\begin{array}{rl}
\displaystyle\inf_{\lambda_\alpha,g_\alpha} &\displaystyle\sum_{\alpha\in\N^n_d}\lambda_\alpha\\
\mbox{s.t.}& \lambda_\alpha-g_\alpha\geq0,\quad\forall\alpha\in\N^n_d\\
& \lambda_\alpha+g_\alpha\geq0,\quad\forall\alpha\in\N^n_d\\
&f(g)\,\leq\,\rho_d;\:\lambda_\alpha\geq0,\quad\forall\alpha\in\N^n_d,
\end{array}\]
which is a convex optimization problem for which Slater's condition holds.
Hence at an optimal solution $(g^*,\lambda)$, the Karush-Kuhn-Tucker (KKT) optimality conditions
conditions read:
\[\begin{array}{rl}
1-u_\alpha-v_\alpha-\psi_\alpha&=0,\quad\forall\alpha\in\N^n_d\\\\
u_\alpha-v_\alpha +\theta\,\frac{\partial f(g^*)}{\partial g_\alpha}&=0,\quad\forall\alpha\in\N^n_d\\\\
\lambda_\alpha,\,u_\alpha,\,v_\alpha,\,\psi_\alpha,\,\theta&\geq0,\quad\forall\alpha\in\N^n_d\\
f(g^*)&\leq\,\rho_d\\
\lambda_\alpha\psi_\alpha\,=\,0;\:u_\alpha\,(\lambda_\alpha-g^*_\alpha)&=0,\quad\forall\alpha\in\N^n_d\\
\theta\,(1-f(g^*))=0;\:v_\alpha\,(\lambda_\alpha+g^*_\alpha)&=0,\quad\forall\alpha\in\N^n_d
\end{array}\]
for some dual variables $(\u,\v,\psi,\theta)$.

The meaning of the above optimality conditions is clear. 
Indeed at an optimal solution $(g^*,\blambda)$ we must have
$\lambda_\alpha=\vert g^*_\alpha\vert$ for all $\alpha$.
Moreover, from the complementarity conditions 
one also has $u_\alpha\,v_\alpha=0$ whenever $g^*_\alpha\neq0$.
In addition, from the two first equations, and the fact that 
$1=u_\alpha+v_\alpha=\vert u_\alpha-v_\alpha\vert$, $\lambda_\alpha=\vert g^*_\alpha\vert$,
all moments $\int_{\G^*}\x^\alpha d\x$ must be equal whenever $g^*_\alpha\neq0$.

We next show that $\x\mapsto g^*(\x)=\sum_{i=1}^nx_i^{d}$ is an optimal solution.
Recall that $\partial{f(g)}/\partial{g_\alpha}=\frac{n+d}{d}\int_{\G^*}\x^\alpha\,d\x$.
Choose
\[\theta\,:=\,\frac{d}{n+d}\left(\int_{\G^*} x_1^{d}d\x\right)^{-1};\quad u_\alpha\,=\,
\theta\,\frac{n+d}{d}\int_{\G^*}\x^\alpha\,d\x,\]
and $v_\alpha=0$, $\lambda_\alpha=g^*_\alpha$, for all $\alpha\in\N^n_d$. 
(Notice that $u_\alpha=1$ whenever $\x^\alpha=x_i^{d}$ for some $i$.)
Hence
$u_\alpha\geq0$ for all $\alpha\in\N^n_d$ (because $u_\alpha=0$ 
whenever some $\alpha_i$ is odd), and let
and $\psi_\alpha:=1-u_\alpha$ for all $\alpha\in\N^n_d$. Observe that
$\psi_\alpha\geq0$ because 
\[\left\vert\int_{\G^*}\x^\alpha\,d\x\,\right\vert\,\leq\, \int_{\G^*}x_1^{d}\,d\x,\qquad\forall\alpha\in\N^n_d.\]
Therefore,
$(g^*,\u,\v,\psi,\blambda,\theta)$ satisfies the (necessary) KKT-optimality conditions 
and as Slater's conditions holds and $\P_1$ is convex, the KKT-optimality conditions are also sufficient.
Hence we may conclude that $g^*$ is an optimal solution of $\P_1$.
Finally, observe that
\begin{eqnarray*}
\ell_1^*\,=\,\Vert g^*\Vert_1\,=\,n\,=\,\sum_{\alpha\in\N^n_d}\lambda_\alpha&=&\sum_{\alpha\in\N^n_d}(u_\alpha-v_\alpha)\,g^*_\alpha\\
&=&-\theta\,\langle\nabla f(g^*),g^*\rangle\,=\,\theta\,\frac{n}{d}f(g^*),
\end{eqnarray*}
from which we deduce
\[\displaystyle\int_{\G^*} x_i^{d}\,d\x\,=\,\frac{1}{n+d}\,\displaystyle\int_{\G^*}\,d\x,\]
for $i=1,\ldots,n$, which is (\ref{relationship}). Finally, for the numerical value of ${\rm vol}\,(\B_d)$ in
(\ref{relationship}) see Lemma \ref{lemma-appendix}.

It remains to prove that $g^*$ above is the unique optimal solution of $\P_1$. So suppose that $\P_1$
has another optimal solution 
$h\in\hom$ (hence such that $h\neq g^*$ and $\Vert h\Vert_1=\Vert g^*\Vert_1=n$). As we have seen, necessarily
$f(h)=f(g^*)=\rho_d={\rm vol}(\B_d)$. But then as $\P_1$ is a convex optimization problem, any convex combination
$h_\lambda:=\lambda h+(1-\lambda)g^*\in\hom$, $\lambda\in (0,1)$, is also an optimal solution of $\P$.
%Thus $f(h_\lambda)=\Vert g^*\Vert_1$ and b
By strict convexity of $f$,
\[f(h_\lambda)\,<\,\lambda f(h)+(1-\lambda)f(g^*)\,=\,\rho_d.\]
But again by taking $k^{-n/d}f(h_\lambda)=\rho_d$ (so that $k<1$) we exhibit another feasible solution
$\tilde{g}:=k\,h_\lambda\in\hom$ with smaller $\ell_1$-norm
norm $\Vert \tilde{g}\Vert_1=k\Vert g^*\Vert_1$, in contradiction with
the fact that $g^*$ is an optimal solution of $\P$. Hence $g^*$ is the unique optimal solution of $\P_1$.
\end{proof}

\subsection*{An alternative formulation}

We may also consider the alternative but equivalent formulation
\begin{equation}
\label{equiv}
\P'_1:\quad\rho'=\inf_g \{\,f(g):\quad \Vert g\Vert_1 \leq n\,;\quad g\in\hom\,\}.
\end{equation}
\begin{prop}
Let $\P_1$ and $\P'_1$ be as in (\ref{problem-P-22}) and (\ref{equiv}), respectively.
Then $\P'_1$ and $\P_1$ have same optimal value $\rho_d$ and moreover the homogeneous polynomial
$\x\mapsto g^*(\x)=\sum_{i=1}^nx_i^{d}$ is the unique optimal solution of $\P'_1$.
\end{prop}
\begin{proof}
Again as $f(\lambda g)=\lambda^{-n/d}f(g)$ we may only consider those $h\in\hom$
with $\Vert h\Vert_1=n$.
Suppose that $h$ is an optimal solution of $\P'_1$ with $f(h)=\rho'<\rho_d$ and $\Vert h\Vert_1=n$.
Then take $\tilde{g}=kh$ with $k^{-n/(d)}\rho'=\rho_d$ so that $k<1$. Then $\Vert \tilde{g}\Vert_1=k\Vert n\Vert_1=kn<n$.
But this implies that $\tilde{g}$ would be a better solution for $\P_1$ than $g^*$, a contradiction. Therefore
$\rho'\geq\rho_d$ and in fact $\rho'=\rho_d$ as $g^*$ is feasible for $\P'_1$ with $f(g^*)=\rho_d$.
Next, observe that $\P'_1$ is a convex optimization problem with a strictly convex objective function;
hence an optimal solution is unique.
\end{proof}
Therefore problem $\P_1$ has the equivalent formulation:
{\it Among all homogeneous polynomials $g\in\hom$ with $\Vert g\Vert_1=1$ which 
is the one with minimum volume?} 
By Theorem \ref{th1} the $L_d$-unit ball has minimum volume.

\section{The $\ell_2$-norm formulation}

Let denote by  $\x\cdot\y$ the usual scalar product in $\R^n$.
For every $\alpha\in\N^n$ let $c_\alpha:=\frac{(\sum_i\alpha_i){\rm !}}{\alpha_1{\rm !}\cdots\alpha_n{\rm !}}$.
Recall that $\N^n_d:=\{\alpha\in\N^n:\vert\alpha\vert=d\}$ and $s_d={n-1+d\choose d}$. We
now write 
\[\x\mapsto p(\x)\,:=\,\sum_{\alpha\in\N^n_d}c_\alpha\,p_\alpha\,\x^\alpha,\qquad p\in\hom,\]
for some vector $\p=(p_\alpha)\in\R^{s_d}$, and  equip $\hom$ with the scalar product
\[\langle p,q\rangle_d\,:=\,\sum_{\alpha\in\N^n_d}c_\alpha \,p_\alpha\,q_\alpha,\quad p,q\in\hom,\]
with associated norm $\Vert p\Vert_{2,d}^2=\langle p,p\rangle_d$.
Next, denote by $\mathcal{P}_d\subset\hom$ the convex cone of homogeneous polynomials of degree $d$
which are nonnegative, and let $\mathcal{C}_d\subset\hom$
be the convex cone of sums of $d$-powers of linear forms. Then $\mathcal{C}_d$
is the dual cone of $\mathcal{P}_d$, i.e.,  $\mathcal{P}_d^*=\mathcal{C}_d$; see e.g. Reznick \cite{reznick}.

As in \S \ref{section-ell-1}, let $\rho_d={\rm vol}(\B_d)$ and  consider the following optimization problem:
\begin{equation}
\label{problem-P-11}
\P_2:\qquad \ell^*_2\,=\,\inf_g\:\{\,\Vert g\Vert_{2,d}^2: \:f(g)\leq\rho_d\,;\quad g\in\hom\,\},
\end{equation}
a (weighted) $\ell_2$-norm analogue of $\P_1$ in (\ref{problem-P-222}).
In view of Theorem \ref{th1}, problem $\P_2$ is a convex 
optimization problem.

\begin{theorem}
\label{th2}
Problem $\P_2$ in (\ref{problem-P-11}) has a unique optimal solution $g^*\in\hom$
whose vector of coefficients $\g^*=(g^*_\alpha)\in\R^{s_d}$
satisfies:
\begin{equation}
\label{th2-1}
g^*_\alpha\,=\,\ell^*_2\,\frac{n+d}{n}\cdot\frac{\displaystyle\int_{\G^*}\x^\alpha\,d\x}{\displaystyle\int_{\G^*}\,d\x},
\qquad \forall\,\alpha\in\N^n_{d},
\end{equation}
where $\G^*=\{\x:g^*(\x)\leq1\}$ and ${\rm vol}\,(\G^*)=\rho_d$.

Therefore, $\g^*=(g^*_\alpha)$, $\alpha\in\N^n_d$, is an element of $\mathcal{C}_d=\mathcal{P}^*_d$. More precisely:
\begin{equation}
\label{lin-form-1}
\x\,\mapsto\,g^*(\x)\,=\,\ell^*_2\,\frac{n+d}{n}\cdot\frac{\displaystyle\int_{\G^*}(\z\cdot \x)^{d}\,d\z}{\displaystyle\int_{\G^*}\,d\z}
\end{equation}
and in fact the exist  $\z_i\in\R^n$, $\theta_i>0$, $i=1,\ldots,s$ with $s\leq {n-1+d\choose d}+1$, such that
\begin{equation}
\label{lin-form-2}
\x\,\mapsto\,g^*(\x)\,=\,\ell^*_2\,\frac{n+d}{n}\,\sum_{i=1}^s \theta_i \,(\z^i\cdot \x)^{d}.
\end{equation}
\end{theorem}
\begin{proof}
That $\P_2$ has an optimal solution follows exactly with same arguments as for $\P_1$.
Moreover Slater's condition also holds for $\P_2$. Hence by the KKT-optimality conditions, there exists $\lambda^*\geq0$ such that
\[2\,g^*_\alpha\,c_\alpha\,=\,-\lambda^*\,\frac{\partial{f(g^*)}}{\partial g_\alpha}\,=\,\lambda^*\,\frac{n+d}{d}\int_{\G^*}c_\alpha\,\x^\alpha\,d\x,\qquad \forall\alpha\in \N^n_{d}.\]
Therefore, multiplying each side with $g^*_\alpha$ and summing up yields
\[2\,\Vert g^*\Vert^2_{2,d}\,=\,2\ell^*_2\,=\,-\lambda^*\langle\nabla f(g^*),g^*\rangle\,=\,\lambda^*\,\frac{n}{d}\,f(g^*)\,=\,\lambda^*\,\frac{n}{d}\,\rho_d.\]
Hence $\lambda^*=4\ell^*_2\,d/(n\rho_d)$ and $g^*_\alpha=\ell^*_2\,\frac{n+d}{n\rho_d}\displaystyle\int_{\G^*}\x^\alpha\,d\x$,
from which (\ref{th2-1}) follows.

But then (\ref{th2-1}) means that $(g^*_\alpha)\in\mathcal{P}_d^*=\mathcal{C}_d$, which yields
(\ref{lin-form-1}). To get (\ref{lin-form-2}) we use a generalization of Tchakaloff's theorem described in
\cite{lowner}, Anastassiou \cite{anastassiou} and Kemperman \cite{kemperman}.
\end{proof}
So both  optimal solutions $g^*_1$ of $\P_1$ and $g^*_2$ of $\P_2$ are sums of $d$-powers of linear forms. 
But $g^*_2$ does not have the parsimony property as shown below.

\begin{cor}
\label{cor-1}
Let $g^*\in\hom$ be the optimal solution of $\P_2$. Then all its coefficients 
$g^*_\alpha$ with $\alpha=2\beta\in\N^n_{d}$ for some $\beta\in\N^n_{d/2}$, are strictly positive. 
Hence if $d\geq4$ the optimal solution $\x\mapsto g^*_1(\x)=\Vert \x\Vert_d^{d}$ of $\P_1$ cannot be an optimal solution.
\end{cor}
\begin{proof}
From the characterization (\ref{th2-1}), every coefficient $g^*_{2\beta}$ with $2\vert\beta\vert=d$
must be positive. Hence the optimal solution $g_1^*$ of problem $\P_1$ cannot be an optimal solution of $\P_2$. 
Moreover there are ${n+d/2-1\choose d/2}$ such coefficients.
\end{proof}
Corollary \ref{cor-1} states that the optimal solution of $\P_2$
does not have a parsimony property as it has at least ${n+d/2-1\choose d/2}$ non zero coefficients!

The only case where the optimal solution of $\P_1$ also solves $\P_2$ is the quadratic case $d=2$.
Indeed straightforward computation shows that 
(\ref{th2-1}) is satisfied by the polynomials $g^*_1$ of Theorem \ref{th1}.

\begin{ex}
{\rm Let $n=2$ and $d=4$. By symmetry we may guess that the optimal solution $g^*\in\hom$ is of the form:
\[\x\mapsto g^*(\x)\,=\,g_{40}\,(x_1^4+x_2^4)\,+\,6\,g_{22}\,x_1^2\,x_2^2,\]
with $\ell^*_2=2g_{40}^2+g_{22}^2$ and
\[g_{40}\,=\,3\,\ell^*_2\,\frac{\displaystyle\int_{\G^*} x_1^4\,d\x}{\displaystyle\int_{\G*}\,d\x}\,;\quad
g_{22}\,=\,3\,\ell^*_2\,\frac{\displaystyle\int_{\G^*} x_1^2\,x_2^2\,d\x}{\displaystyle\int_{\G^*}\,d\x}.\]
%In particular we must have $g_{40}>g_{22}$. 
Observe that by homogeneity, an optimal solution $g^*$ of $\P_2$
is also optimal when we replace $\rho_d$ with any constant $a$; only the optimal value changes
and the characterization (\ref{th2-1}) remains the same with the new optimal value $\ell^*_2$. 
After several numerical trials we conjecture that
\[\x\mapsto g^*(\x)\,\approx\,x_1^4+x_2^4+2\,x_1^2x_2^2\,=\,(x_1^2+x_2^2)^2,\]
i.e. $\g^*=(1,0,1/3,0,1)$, is an optimal solution. But then observe that
\[\G^*\,=\,\{\,\x:g^*_2(\x)\leq1\,\}\,=\,\{\,\x:(x_1^2+x_2^2)^2\leq\,1\,\}\,=\,\B_2!\]
%Its  level set looks very much like a sphere (as displayed in Figure \ref{fig-sphere4}). And indeed, using the method proposed in \cite{sirev},
That is, $\G^*$ is another representation of the unit sphere $\B_2$ by
homogeneous polynomials of degree $4$ instead of quadratics!
\[\int_{\G^*}\,d\x\,\approx\,3.1415926\,;\quad \int_{\G^*}x_1^4\,d\x\,\approx\,0.392699\,;\quad
\int_{\G^*}x_1^2x_2^2\,d\x\,\approx\,0.130899.\]
With $a:=\int_{\G^*}d\x$, (\ref{th2-1}) yields  (up to $10^{-8}$)
\[3\,\ell^*_2\cdot\frac{\displaystyle\int_{\G^*}\x_1^4\,d\x}{\displaystyle\int_{\G^*}\,d\x}\,=\, 1= g^*_{40}\,;
\quad 3\ell^*_2\cdot\frac{\displaystyle\int_{\G^*}x_1^2x_2^2\,d\x}{\displaystyle\int_{\G^*}\,d\x}\,=\, 
\frac{1}{3}\,=\,g^*_{22}.\]
Observe also that 
\[(x_1^2+x_2^2)^2\,=\,\frac{1}{6}(x_1+x_2)^4+\frac{1}{6}(x_1-x_2)^4+\frac{2}{3}\,(x_1^4+
x_2^4),\]
i.e., a sum of $4$-powers of linear forms as predicted by Theorem \ref{th2}.
}\end{ex}
In fact we have:
\begin{theorem}
If $d=4$ then for every $n$ the optimal solution of $\P_2$ 
is $g^*(\x)=(\sum_{i=1}^n x_i^2)^2$ whose level set $\G$ is the unit ball $\B_2$.
\end{theorem}
\begin{proof}
Let $\x\mapsto g^*(\x):=(\sum_{i=1}^n x_i^2)^2$. It is enough to prove that (\ref{th2-1}) holds 
(as by homogeneity (\ref{th2-1}) still holds when one replaces $g$ with $\lambda\,g$ for any 
$\lambda>0$.)
Since $g^*(\x)=\sum_ix_i^4+2\sum_{i<j}x_i^2x_j^2$, we have $\ell^*_2=(n+\frac{n(n-1)}{2}\cdot 6(\frac{2}{6})^2)=n(n+2)/3$.
Moreover
\[A:=\int_{\B_2}x_1^4\,d\x\,\left(\int_{\B_2}d\x\,\right)^{-1}\,=\,\frac{3}{(n+4)(n+2)}\]
Therefore with $d=4$,
\[\ell^*_2\,\frac{n+d}{n}\,A\,=\,\frac{n(n+2)}{3}\,\frac{n+4}{n}\,\frac{3}{(n+2)(n+4)}\,
\,=\,1\,=\,g^*_{40}.\]
Similarly one has
\[B:=\int_{\B_2}x_1^2\,x_2^2\,d\x\,\left(\int_{\B_2}d\x\,\right)^{-1}\,=\,\frac{1}{(n+4)(n+2)}\]
so that
\[\ell^*_2\,\frac{n+4}{n}\,B\,=\,\frac{n(n+2)}{3}\,\frac{n+4}{n}\,\frac{1}{(n+2)(n+4)}\,
\,=\,\frac{1}{3}\,=\,g^*_{22}.\]
\end{proof}
In other words, when $d=4$ the Euclidean unit ball $\B_2=\{\x:\sum_ix_i^2\leq1\}$ (which has
the equivalent quartic representation $\{\x: (\sum_{i=1}^n x_i^2)^2\leq 1\}$) solves problem $\P_2$!

\section{The SOS formulation}

As we have seen that both optimal solutions of the $\ell_1$-norm and $\ell_2$-norm formulations are 
sums of $d/2$ powers of linear forms, hence sums of squares (in short SOS). Therefore one may now restrict to
homogeneous polynomials in $\hom$ that are SOS, i.e., polynomials of the form
\[\x\,\mapsto\,g_\Q(\x)\,=\,\v_{d/2}(\x)^T\Q\,\v_{d/2}(\x),\]
where $\v_{d/2}(\x)=(\x^\alpha)$, $\alpha\in\N^n_{d/2}$, and $\Q$ is some real psd symmetric matrix 
($\Q\succeq0$) of size $s(d/2)={n-1+d/2\choose d/2}$. If we denote by $\s_d$ the space of real symmetric matrices of size $s(d)$, there is not a one-to-one correspondence between 
$g\in\hom$ and $\Q\in\s_d$ as several $\Q$ may produce the same polynomial $g_\Q$. 

Given $0\preceq\Q\in\s_d$, denote by $\G_\Q$ the sublevel set
$\{\x:g_\Q(\x)\,\leq\,1\}$ associaterd woth $g_\Q\in\hom$ and let $f(\Q):={\rm vol}\,(\G_\Q)$. Observe that 
again $f$ is positively homogeneous of degree $-n/d$.\\

So the natural analogue for $\Q$ of the $\ell_1$-norm $\Vert g\Vert_1$ for $g\in\hom$
is now the nuclear norm of $\Q$ which as $\Q\succeq0$ reduces to $\langle \bI,\Q\rangle={\rm trace}\,(\Q)$.
It is well-known that optimizing the nuclear norm on convex problems 
with matrices induce a parsimony effect, namely an optimal solution 
will generally have  a small rank. In our context, $\Q$ having a small
rank means that $g_\Q$ can be written a sum of a small number of squares.
However, when expanded in the monomial basis, $g_\Q$ may have many non-zero coefficients
and so its $\ell_1$-norm $\Vert g_\Q\Vert_1$ may not be small.\\

So in the same spirit as for the $\ell_1-$ and $\ell_2$-norm, we now consider the optimization problem:
\[\P_3:\quad \ell^*_3=\inf_{\Q\in\s_d}\,\{\,\langle \bI,\Q\rangle:\:f(\Q)\,\leq\,\rho_d\,; \quad\Q\succeq0\,\},\]
and characterize its unique optimal solution $\Q^*$.
\begin{theorem}
\label{th-Q}
Problem $\P_3$ has a unique optimal solution. Moreover, $\Q^*\in\s_d$
is an optimal solution of $\P_3$ if and only if $f(\Q^*)=\rho_d$ and:
\begin{equation}
\label{th-Q-1}
\bI\,\succeq\,\frac{(n+d)\,\langle\bI,\Q^*\rangle}{n\,\rho_d}\,\displaystyle\int_{\G_{\Q^*}}\v_{d/2}(\x)\,\v_{d/2}d(\x)^T\,d\x
%{\displaystyle\int_{\G_{\Q^*}}\,d\x},
\end{equation}
where
\[\x\mapsto g_{\Q^*}(\x):=\v_{d/2}(\x)^T\Q^*\v_{d/2}(\x),\quad\x\in\R^n.\]
Moreover, the polynomial $\x\mapsto g(\x)=\sum_{i=1}^nx_i^{d}$ cannot be solution
of $\P_3$.
\end{theorem}
\begin{proof}
Let $(\Q_k)$, $k\in\N$, be a minimizing sequence. As $\sup_k\langle \bI,\Q_k\rangle \leq\langle \bI,\Q_1\rangle$,
the sequence $(\Q_k)$ is norm-bounded in $\s_d$. Therefore
it has a converging subsequence $\Q_{k_j}\to\Q^*\in\s_d$ with $\Q^*\succeq0$
and $f(\Q^*)\leq\rho_d$ (by a simple continuity argument). Therefore $\Q^*$ is an optimal solution
of $\P_3$ and again, uniqueness follows from the strict convexity and homogeneity of the function $f$.
Moreover, Slater's condition obviously holds for $\P_3$ which is a convex optimization problem.
Then the  KKT-optimality conditions read:
\begin{eqnarray}
\label{kkt1}
\bI+\lambda\,\nabla f(\Q^*)&=&\Psi\,\succeq\,0\\
\label{kkt2}
\langle \Q^*,\Psi\rangle&=&0\\
\label{kkt3}
f(\Q^*)&\leq&\rho_d; \quad \lambda\,(f(\Q^*)-\rho_d)=0,
\end{eqnarray}
for some dual variables $(\lambda,\Psi)\in\R_+\times\s_d$.
By homogeneity one must have $f(\Q^*)=\rho_d$.
Again Euler's identity for homogeneous functions yields $\langle \nabla f(\Q^*),\Q^*\rangle=-nf(\Q^*)/d$. 
Therefore using (\ref{kkt1}) and (\ref{kkt2}) 
one obtains $\ell^*_3=\langle \bI,\Q^*\rangle=\lambda\,n\,\rho_d/d$, that is, $\lambda=d\,\ell^*_3/(n\rho_d)$.
Next combining $\Psi\succeq0$ with Theorem \ref{th1}, one also gets
\[\bI\,\succeq\,\frac{n+d}{d}\,\lambda \int_{\G_{\Q^*}}\v_{d/2}(\x)\,\v_{d/2}(\x)^T\,d\x,\]
or, equivalently
\[\bI\,\succeq\,\frac{(n+d)\,\langle\bI,\Q^*\rangle}{n\,\rho_d}\,\displaystyle\int_{\G_{\Q^*}}\v_{d/2}(\x)\,\v_{d/2}(\x)^T\,d\x,\]
which is (\ref{th-Q-1}). This proves the {\it only if part} in Theorem \ref{th-Q}.

Conversely, assume that $0\preceq\Q^*\in\s_d$ satisfies $f(\Q^*)=\rho_d$ and (\ref{th-Q-1}).
Let
\[\lambda\,:=\,d\,\langle\bI,\Q^*\rangle/(n\rho_d)\,;\quad\Psi:=\bI\,-\,
\frac{(n+d)\,\langle\bI,\Q^*\rangle}{n\,\rho_d}\,\displaystyle\int_{\G_{\Q^*}}\v_{d/2}(\x)\,\v_{d/2}(\x)^T\,d\x.\]
Obviously $\lambda\geq0$, $\Psi\succeq0$, and:
\begin{eqnarray*}
\langle \Q^*,\Psi\rangle&=&
\langle\bI,\Q^*\rangle\left[1-\frac{n+d}{n\,\rho_d}\,
\displaystyle\int_{\G_{\Q^*}}\left\langle\Q^*,\v_{d/2}(\x)\,\v_{d/2}(\x)^T\right\rangle\,d\x\right]\\
&=&\langle\bI,\Q^*\rangle\left[1-\frac{n+d}{n\,\rho_d}\,
\displaystyle\int_{\G_{\Q^*}}g_{\Q^*}(\x)\,d\x\right]\\
&=&\langle\bI,\Q^*\rangle\left[1-\frac{n+d}{n\,\rho_d}\,
\frac{n}{n+d}\displaystyle\int_{\G_{\Q^*}}\,d\x\right]\quad\mbox{[by Theorem \ref{th1}]}\\
&=&0,
\end{eqnarray*}
which shows that the triplet $(\Q^*,\lambda,\Psi)$ satisfy the KKT-optimality conditions (\ref{kkt1})-(\ref{kkt3}).
As Slater's condition holds for $\P_3$, (\ref{kkt1})-(\ref{kkt3}) are sufficient for optimality, which  concludes the {\it if part}
of the proof.

We next prove that $\x\mapsto g(\x):=\sum_{i=1}^n x_i^d$ (so that $\G=\B_d$) cannot be the optimal solution of $\P_3$. Among all $\Q\succeq0$ such that $\sum_ix_i^d=\v_{d/2}(\x)^T\Q\v_{d/2}(\x)$, the one that minimizes 
${\rm trace}\,(\Q)$ is $\Q=\I$ with ${\rm trace}\,(\Q)=n$.
By Lemma \ref{lemma-appendix}, observe that
\[\frac{\displaystyle\int_{\G}x_i^d\,d\x}{\displaystyle\int_{\G}d\x}\,=\,
\frac{\displaystyle\int_{\B_d}x_i^d\,d\x}{\rho_d}\,=\,\frac{1}{n+d},\]
and so (\ref{th-Q-1}) cannot hold because for instance the north-west and south-east corner
elements of the matrix
\begin{eqnarray*}
\A&:=&\bI\,-\,\frac{(n+d)\,\langle\bI,\Q\rangle}{n\,\rho_d}\,\displaystyle\int_{\G}\v_{d/2}(\x)\,\v_{d/2}d(\x)^T\,d\x\\
&=&\bI\,-\,(n+d)\frac{\displaystyle\int_{\G}\v_{d/2}(\x)\,\v_{d/2}d(\x)^T\,d\x}
{\displaystyle\int_{\G}\,d\x}\end{eqnarray*}
vanish whereas the north-east and south-west corner elements are non-zero, in
contradiction with $\A\succeq0$. 
\end{proof}
The fact that the $L_d$-unit ball is not an optimal solution of $\P_3$
is not a surprise as  the sparsity-induced norm ${\rm trace}\,(\Q)$ (when $\Q\succeq0$)
aims at find a polynomial $g_\Q\in\hom$ which can be written as a sum of squares
with as few terms as possible in the sum. On the other hand, the sparsity-induced norm
$\Vert g\Vert_1$ aims at finding a polynomials $g\in\hom$ with
as few monomials as possible when $g$ is expanded in the monomial basis. These can be two conflicting criteria!

\section{Extension to generalized polynomials}

In this section $d$ is now a (positive) rational with $L_d$-unit ball
$\{\,\x:\:\,\sum_{i=1}^n\vert x_i\vert^d\leq1\,\}$. Even though the function $\x\mapsto \sum_{i=1}^n\vert x_i\vert^d$ is a ``generalized polynomial" and not a polynomial any more,  it is still a nonnegative positively homogeneous of degree $d$ for which Theorem \ref{th1}(a) applies. On the other hand, the vector
space of positively homogeneous functions of degree $d$ is not finite-dimensional and so for optimization purposes we need define an appropriate finite-dimensional
analogue of $\hom$.

We will use the notation $\vert\x\vert\in\R^n_+$ for the vector $(\vert x_1\vert,\ldots,\vert x_n\vert)$ and
$\vert\x\vert^\alpha$ for the generalized monomial $\vert x_1\vert^{\alpha_1}\cdots\vert x_n\vert^{\alpha_n}$, whenever
$\alpha\in\mathbb{Q}^n_+$.
\begin{defn}
Let $0<d\in\mathbb{Q}$. Define the space $\mathscr{C}_d$ as:
\begin{equation}\label{general}
\mathscr{C}_d\,:=\,\{\,\sum_{\alpha\in\mathbb{Q}^n_+}g_\alpha\,\vert \x\vert^\alpha\,:\quad g_\alpha\in\R\,;\quad\vert\alpha\vert\,(:=\sum_{i=1}^n \alpha_i)\,=d\,\}
\end{equation}
where only {\it finitely many} coefficients $g_\alpha$ are non-zero. Then:
\[\Vert g\Vert_1\,=\,\sum_\alpha\,\vert g_\alpha\vert\,;\quad
\Vert g\Vert_2^2\,=\,\sum_\alpha\,g_\alpha^2.\]
\end{defn}
The space $\mathscr{C}_d$ is a real infinite-dimensional vector space and each element of $\mathscr{C}_d$ is a
positively homogeneous functions of degree $d$.
\begin{defn}
With $q\in\N$ let $\Z^n_q\subset\R^n$ be the lattice $\{\,\z\in\R^n: q\,\z\in\Z^n\,\}$. With
$0<d\in\Z_q$, denote by
$\n^n_{dq}$ the finite set $\{\,\alpha\in\Z^n_q:\alpha\geq0\,;\:\sum_{i=1}^n\alpha_i=d\,\}$
of cardinality $m(d,q)$ and by $\homq\subset\mathscr{C}_d$
the vector space of functions $g:\R^n\to \R$ defined by:
\begin{equation}
\label{help}
\homq\,:=\,\{\sum_{\alpha\in\n^n_{dq}}\,g_\alpha\,\vert \x\vert^{\alpha}:\quad (g_\alpha)\in\R^{m(d,q)}\,\},\end{equation}
which is a finite-dimensional vector space.
\end{defn}
For instance with $n=2$ and $d=1$ one has $\B_1:=\{\,\x\in\R^n:\:\sum_{i=1}^n\vert x_i\vert\leq1\,\}$ and with $0<q\in\N$,
\[\n_{1q}\,=\,\{(\frac{k}{q},\frac{q-k}{q}): k=0,\ldots,q\,\}\,;\quad m(d,q)=q+1,\]
and $g\in\homq$ can be written as
\[\x\mapsto g(\x)\,=\,\sum_{k=0}^q\,g_k \,\vert x_1\vert  ^{\frac{k}{q}}\,\vert x_2\vert^{\frac{q-k}{q}},\qquad \x\in\R^2,\]
for some vector $\g=(g_k)\in\R^{t+1}$. 

Obviously
$\homq$ is a vector space of dimension $m(d,q)$ and every function in $\homq$ is positively homogeneous
of degree $d$. 
As we did in \S \ref{section-ell-1}, with $g\in\homq$ is associated the level set $\G:=\{\,\x:g(\x)\leq 1\,\}$, and 
whenever $\G$ has finite volume let $f(g):={\rm vol}\,(\G)$, $g\in\homq$. It follows that
$f$ is positively homogeneous of degree $-n/d$. Therefore (\ref{vol-exp})-(\ref{relation-2})
in Theorem \ref{th1} holds for $f$. 
Finally, let
\[\x\mapsto g^*(\x)\,:=\,\sum_{i=1}^n\vert x_i\vert^d;\quad \B_d=\G^*;\quad \rho_d\,=\,{\rm vol}\,(\G^*),\]
and consider the finite-dimensional optimization problem
\begin{equation}
\label{problem-P-33}
\P_{1q}:\qquad \ell_1^*\,=\,\inf_g\:\{\,\Vert g\Vert_1: \:f(g)\,\leq\,\rho_d\,;\quad g\in\homq\,\}.
\end{equation}
When $d<1$ the unit ball $\B_d$ is not convex and is not associated with a norm as can be seen
in Figure \ref{fig-undemi} where $d=1/2$.
\begin{center} 
\begin{figure}
 \resizebox{0.9\textwidth}{!}
{\includegraphics{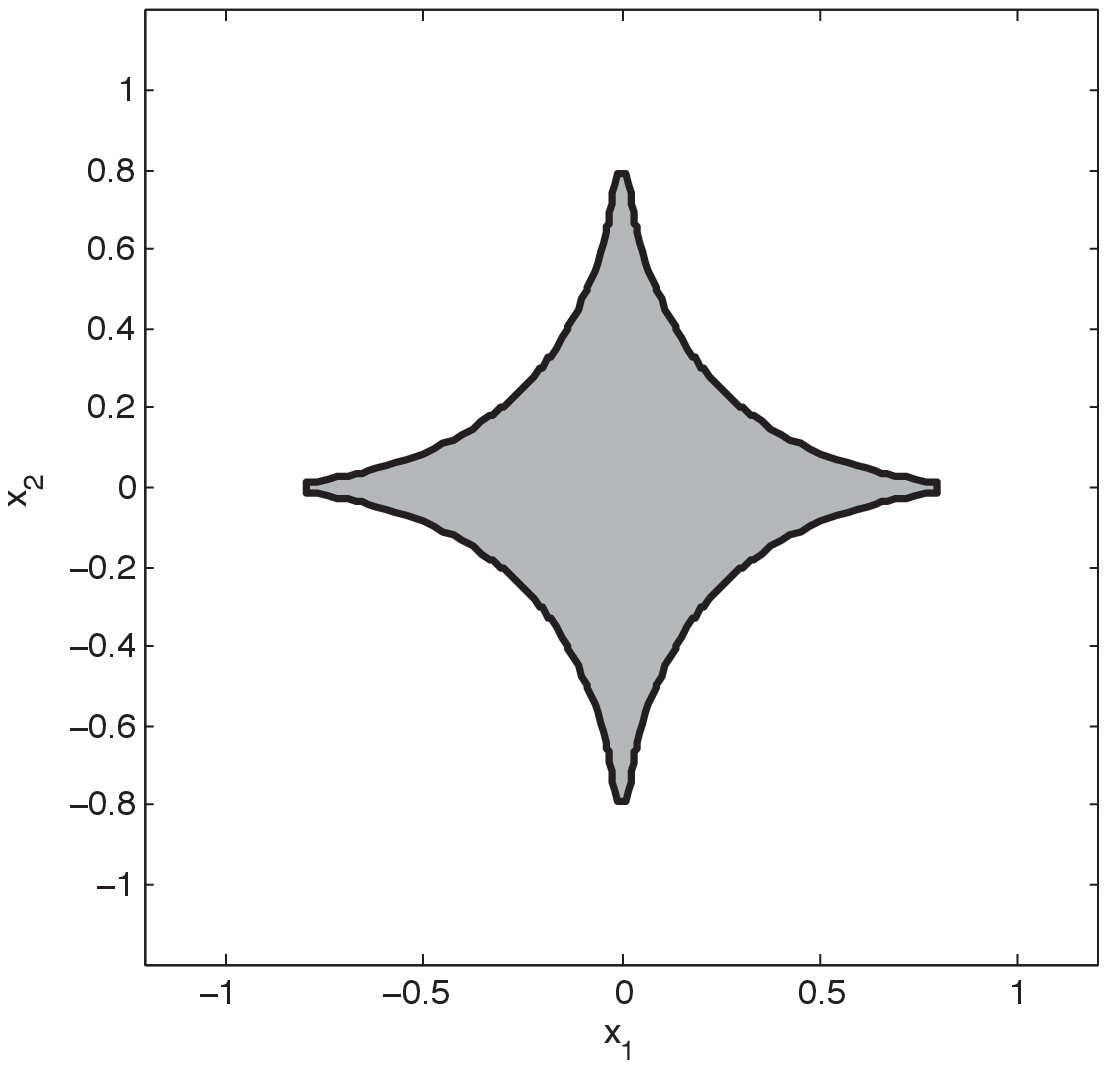}}
\caption{The $L_{1/2}$-unit ball $\{\,\x:\sqrt{\vert x_1\vert}+\sqrt{\vert x_2\vert}\leq1\,\}$ \label{fig-undemi}}
\end{figure}
\end{center} 

However, we have the following analogue of Theorem \ref{th2-ell1}

\begin{theorem}
\label{th3-ell1-q}
Let $0<q\in\N$ and $0<d\in\Z_{q/2}$. The polynomial
\begin{equation}
\label{general-gstar}
\x\mapsto g^*(\x)\,=\,\sum_{i=1}^n \vert x_i\vert ^{d}\,\end{equation}
is the unique optimal solution of Problem $\P_{1q}$ in (\ref{problem-P-33}) and
moreover,
\begin{equation}
\label{relationship-q}
{\rm vol}\,(\B_d)\,=\,\frac{2^n\Gamma(1/d)^n}{n\,d^{n-1}\,\Gamma(n/d)};
\:\displaystyle\int_{\G^*} \vert x_i\vert ^{d}\,d\x
\,=\,\frac{{\rm vol}\,(\B_d)}{n+d}\,\quad i=1,\ldots,n.
\end{equation}
\end{theorem}
\begin{proof}
The proof is almost a verbatim copy of that of Theorem \ref{th2-ell1} except that we now have to deal with generalized moments $\int_\G \vert\x\vert^\alpha d\x$, $\alpha\in\n^n_{dq}$ instead of standard
monomial moments $\int_\G \x^\alpha d\x$, $\alpha\in\N^n$. But the crucial fact that we exploit is that $f$ is strictly convex and Theorem \ref{th1} holds for $f$. As in the proof of Theorem \ref{th2-ell1}, to show
that $g^*$ in (\ref{general-gstar}) satisfies the KKT-optimality conditions we only need prove that
\[\int_{\G^*}\vert\x\vert^\alpha\,d\x\,\leq\, \int_{\G^*}\vert x_1\vert^{d}\,d\x,\qquad\forall\alpha\in\n^n_{dq}.\]
Define the Hankel-type moment matrix $\M$ to be the real symmetric matrix with rows
and columns indexed by $\alpha\in\n^n_{d\frac{q}{2}}$ and with entries
\[\M(\alpha,\beta)\,:=\,\int_{\G^*}\vert\x\vert^{(\alpha+\beta)/2}\,d\x,\qquad\alpha,\beta\,\in\,\n^n_{d\frac{q}{2}}.\]
Equivalently, letting $\N^n_{dq/2}=\{\beta\in\N^n:\sum_i\beta_i=dq/2\}$ and re-indexing rows and columns of $\M$ with $\tilde{\alpha}:=q\alpha/2\in\N^n_{dq/2}$,
\[\M(\tilde{\alpha},\tilde{\beta})\,:=\,\int_{\G^*}(\vert\x\vert^{1/q})^{\tilde{\alpha}+\tilde{\beta}}\,d\x\,=:\,y_{\tilde{\alpha}+\tilde{\beta}},\qquad\tilde{\alpha},\tilde{\beta}\,\in\,\N^n_{d\frac{q}{2}}.\]
Define $\y=(y_{\tilde{\alpha}})$, $\tilde{\alpha}\in\N^n_{dq}$, and $\tilde{X}=\vert\x\vert^{1/q}$.
Observe that from (\ref{help}) one may write
\[\homq\,:=\,\{\sum_{\tilde{\alpha}\in\N^n_{dq}}\,g_\alpha\,(\vert \x\vert^{\frac{1}{q}})^{\tilde{\alpha}}:\quad (g_\alpha)\in\R^{m(d,q)}\,\}.\]
Let $L_\y:\homq\to\R_+$ be the linear mapping defined by
\[g\mapsto\,L_\y(g)\,:=\,\sum_{\tilde{\alpha}\in\N^n_{dq}}g_\alpha\,y_{\tilde{\alpha}}
\,=\,\sum_{\tilde{\alpha}\in\N^n_{dq}}g_\alpha\,\int_{\G^*}\tilde{X}^{\tilde{\alpha}}\,d\x\,=\,
\int_{\G^*}g(\x)\,d\x.\]
By an adaptation of Lemma 4.3 in Lasserre and Netzer \cite{netzer} 
to the present homogeneous context one has
\[\vert y_{\tilde{\alpha}}\vert\,\leq\,\sup_{i=1,\ldots,n}\,L_\y(\tilde{X}_i^{dq})\,=\,\sup_{i=1,\ldots,n}\,
\int_{\G^*}\vert x_i\vert^{d}\,d\x\,\left(=\,\int_{\G^*}\vert x_1\vert^{d}\,d\x\right),\quad\forall\tilde{\alpha}\in\N^n_{dq}.\]
Indeed in Lemma 4.3 of \cite{netzer} one only uses the Hankel structure of the moment matrix $\M$
and its positive definiteness. Therefore for every $\alpha\in\n^n_{dq}$,
\[\int_{\G^*}\vert\x\vert^\alpha\,d\x\,=\,\int_{\G^*}\tilde{X}^{\tilde{\alpha}}\,d\x\,=\,
 y_{\tilde{\alpha}}\,\leq\,\int_{\G^*}\vert x_1\vert^{d}\,d\x,\]
 and so as in the proof of Theorem \ref{th2-ell1}, $g^*$ satisfies the KKT-optimality conditions.
\end{proof}

We obtain the following even more general extension of Theorem \ref{th2-ell1}.
\begin{cor}
\label{cor2}
For every $0<d\in\mathbb{Q}$ the generalized polynomial
\[\x\mapsto g^*(\x)\,=\,\sum_{i=1}^n \vert x_i\vert ^{d}\,\]
is the unique optimal solution of 
\[\P_1:\qquad \ell_1^*\,=\,\inf_g\:\{\,\Vert g\Vert_1: \:f(g)\,\leq\,\rho_d\,;\quad g\in\mathscr{C}_d\,\}.\]
and (\ref{relationship-q}) holds.
\end{cor}
\begin{proof}
Let $0<d\in\mathbb{Q}$ and suppose that there exists 
$g\in\mathscr{C}_d$ such that
${\rm vol}\,(\G)=\rho_d$ and $\Vert g\Vert_1\leq n$. Write $d=p_0/q_0$
with $0<p_0,q_0\in\N$. For each non-zero coefficient
$g_\alpha$ one has $\alpha_i=p_i(\alpha)/q_i(\alpha)$ for some integers $0<p_i(\alpha),q_i(\alpha)$. Let
$q=2q'$ with $q'\in\N$ being the least common multiple (l.c.m.) of $\{q_0,(q_i(\alpha)), i=1,\ldots,n, g_\alpha\neq0\}$.
Then  $d\in\Z_{q/2}$ and $g\in\homq$. Therefore by Theorem \ref{th3-ell1-q},
$\Vert g\Vert_1>\Vert g^*\Vert_1=n$ where $g^*(\x)=\sum_{i=1}^n \vert x_i\vert^d$, in contradiction  
with our assumption $\Vert g\Vert_1\leq n$.
\end{proof}

Then again the parsimony property of the $L_d$-unit ball $\B_d$ 
can be retrieved by minimizing the $\ell_1$-norm 
over all nonnegative generalized polynomials $g\in\mathscr{C}_d$ whose associated ball
unit ball $\G$ has finite volume.\\

Next, concerning the $\ell_2$-norm, with $0<q\in\N$ an analogue of problem (\ref{problem-P-11}) now reads:
 \begin{equation}
\label{problem-P-11-q}
\P_{2q}:\qquad \ell^*_2\,=\,\inf_g\:\{\,\Vert g\Vert_2^2: \:f(g)\leq\rho_d\,;\quad g\in\homq\,\},
\end{equation}
and we have the following analogue of Theorem \ref{th2}:

 \begin{theorem}
\label{th2-q}
With $0<q\in\N$ and $0<d\in\Z_q$, Problem $\P_{2q}$ in (\ref{problem-P-11-q}) has a unique optimal solution $g^*\in\homq$
whose vector of coefficients $\g^*=(g^*_\alpha)\in\R^{m(d,q)}$
satisfies:
\begin{equation}
\label{th2-1-q}
g^*_\alpha\,=\,\ell^*_2\,\frac{n+d}{n}\cdot\frac{\displaystyle\int_{\G^*}\vert\x\vert^\alpha\,d\x}{\displaystyle\int_{\G^*}\,d\x},
\qquad \forall\,\alpha\in\n_{dq},
\end{equation}
where $\G^*=\{\x:g^*(\x)\leq1\}$ and ${\rm vol}\,(\G^*)=\rho_d$.
\end{theorem}
We omit the proof as it is again a verbatim copy of that of Theorem \ref{th2}.
But in contrast to the case of polynomials in Theorem \ref{th2},
in the optimal solution $g^*$ of $\P_{2q}$, all coefficients $(g^*_\alpha)$, $\alpha\in\n_{dq}$, 
are non-zero! This follows from (\ref{th2-1-q}) and the fact that 
all generalized moments
$\int_{\G^*}\vert\x\vert^\alpha\,d\x$, $\alpha\in\n_{dq}$ are non-zero! For instance with
$d=1/2$ and $q=8$,
\[0<\int_\G \vert x_1\vert^{1/2}\,d\x\,;\quad
0<\int_\G \vert x_1\vert^{1/8}\vert x_2\vert^{3/8}\,d\x\,;\quad
0<\int_\G \vert x_1\vert^{1/4}\vert x_2\vert^{1/4}\,d\x,\]
\[0<\int_\G \vert x_1\vert^{3/8}\vert x_2\vert^{1/8}\,d\x\,;\quad
0<\int_\G \vert x_2\vert^{1/2}\,d\x.\]
Hence the unique optimal solution $g^*$ of $\P_{2q}$ is not sparse at all. Even more,
with fixed $0<d\in\mathbb{Q}$, the larger is $q$ the more complicated is $g^*$! Therefore
an analogue of Corollary \ref{cor2} for the $\ell_2$-norm cannot exist.

\section{Appendix}

\begin{lem}
\label{lemma-appendix}
Let $d$ be a positive real and let $g:\R^n\to \R$ be the function:
\[\x\mapsto g(\x)\,:=\,\sum_{i=1}^n\vert x_i\vert^d\,;\quad \G:=\{\,\x\in\R^n:g(\x)\leq\,1\,\}.\]
Then:
\begin{equation}
\label{apen-1}
\int_{\G}\,d\x\,=\,\frac{2^n}{n\,d^{n-1}}\,\frac{\Gamma(1/d)^n}{\Gamma(n/d)};
\quad\int_{\G}\vert x_i\vert^d \,d\x\,=\,
\frac{2^n}{n\,(n+d)\,d^{n-1}}\,\frac{\Gamma(1/d)^n}{\Gamma(n/d)},
\end{equation}
for all $i=1,\ldots,n$.
\end{lem}
\begin{proof}
The function $g$ is positively homogeneous of degree $d$.
Observe that by Theorem \ref{th1}, $\displaystyle\int_{\G}d\x=\frac{1}{\Gamma(1+n/d)}\displaystyle\int_{\R^n}\exp(-g(\x))\,d\x$, and
again by Theorem \ref{th1},
\begin{eqnarray*}
\int_{\R^n}\exp(-g(\x))\,d\x&=&\left(\int_\R\exp(-\vert t\vert ^d)\,dt\right)^n\,=\,\left(\Gamma(1+\frac{1}{d})\int_{\vert t\vert^d\leq1}dt\right)^n
=\left(\frac{2\Gamma(1/d)}{d}\right)^n
\end{eqnarray*}
where we have used the identity $x\Gamma(x)=\Gamma(1+x)$. This yields the result in the 
left of (\ref{apen-1}). Similarly, for every $i=1,\ldots,n$,
\begin{eqnarray*}
\int_{\R^n}\vert x_i\vert^d\,\exp(-g(\x))\,d\x&=&\left(\int_\R \vert t\vert^d\,\exp(-\vert t\vert^d)\,dt\right)\,\left(\int_\R\exp(-\vert t\vert^d)\,dt\right)^{n-1}\\
&=&\left(\Gamma(1+\frac{d+1}{d})\int_{\vert t\vert^d\leq1}\vert t\vert ^d \,dt\right)\left(\Gamma(1+\frac{1}{d})\int_{\vert t\vert^d\leq1}dt\right)^{n-1}\\
&=&\left(\frac{2}{d}\Gamma((d+1)/d)\right)\left(\frac{2}{d}\Gamma(1/d)\right)^{n-1}\\
&=&\frac{2^n}{d^{n+1}}\Gamma(1/d)^n.
\end{eqnarray*}
Therefore,
\[\int_{\G}\vert x_i\vert^d \,d\x\,=\,\frac{1}{\Gamma(1+\frac{n+d}{d})}\,\int_{\R^n}\vert x_i\vert^{d}\exp(-g(\x))\,d\x
\,=\,\frac{2^n}{n\,(n+d)\,d^{n-1}}\,\frac{\Gamma(1/d)^n}{\Gamma(n/d)}.\]
\end{proof}

\end{document}